\documentclass[12pt]{article}
\usepackage{graphicx}
\usepackage{amsmath,amsthm,amssymb,amsfonts,euscript,enumerate}
\usepackage{amsthm}
\usepackage{color,wrapfig}
\usepackage{pxfonts}
\usepackage{verbatim}
\newtheorem{thm}{Theorem}[section]

\newcommand{\R}{{\mathbb{R}}}

\newcommand{\1}{\partial}

\begin{document}
\title{A note on compact gradient Yamabe solitons} 
\author{Shu-Yu Hsu\\
Department of Mathematics\\
National Chung Cheng University\\
168 University Road, Min-Hsiung\\
Chia-Yi 621, Taiwan, R.O.C.\\
e-mail: syhsu@math.ccu.edu.tw}
\date{July 19, 2011}
\smallbreak \maketitle
\begin{abstract}
We will give a simple proof that the metric of any compact Yamabe 
gradient soliton $(M,g)$ is a metric of constant scalar curvature
when the dimension of the manifold $n\ge 3$.
\end{abstract}

\vskip 0.2truein


\vskip 0.2truein
\setcounter{equation}{0}
\setcounter{section}{-1}

\setcounter{equation}{0}
\setcounter{thm}{0}

Recently there is a lot of study of the Yamabe flow on manifolds by 
S.~Brendle \cite{B1}, \cite{B2}, B.~Chow \cite{C}, P.~Daskalopoulos and 
N.~Sesum \cite{DS}, S.Y.~Hsu \cite{H}, A.~Burchard, R.J.~Mccan and A.~Smith
\cite{BMS}, L.~Ma and L.~Cheng \cite{MC}, M.~Del Pino, M.~S\'aez \cite{PS}
and others. A time dependent metric $g(\cdot,t)$ on a Riemannian manifold 
$M$ is said to evolve by the Yamabe flow if the metric $g$ satisfies
$$
\frac{\1 }{\1 t}g_{ij}=-Rg_{ij}
$$
on $M$ where $R$ is the scalar curvature. Yamabe gradient solitons are special 
solutions of Yamabe flow. We say that a metric $g_{ij}$ on a Riemannian 
manifold $M$ a Yamabe gradient soliton if there exists a smooth function 
$f:M\to\R$ and a constant $\rho\in\R$ such that
\begin{equation}\label{soliton-defn}
(R-\rho)g_{ij}=\nabla_i\nabla_jf\qquad\mbox{ on }M.
\end{equation}
It is proved in \cite{DS} the metric of any compact Yamabe gradient
soliton $(M,g)$ is a metric of constant scalar curvature. In this paper we will
give a simple alternate proof of this interesting result.

The main theorem of this paper is the following.

\begin{thm}
Let $(M,g)$ be a $n$-dimensional compact Yamabe gradient soliton with $n\ge 3$.
Then $(M,g)$ is a manifold
of constant scalar curvature.
\end{thm}
\begin{proof}
As observe in P.20 of \cite{DS} (cf. \cite{C}) by \eqref{soliton-defn} and a 
direct computation one has
\begin{equation}\label{R-eqn}
(n-1)\Delta_gR+\frac{1}{2}<\nabla R,\nabla f>_g+R(R-\rho)=0.
\end{equation}
Tracing \eqref{soliton-defn} over $i,j$,
\begin{align}
&n(R-\rho)=\Delta f\label{f-laplacian-eqn}\\
\Rightarrow\quad&\int_M(R-\rho)\,dV=\frac{1}{n}\int_M\Delta f\,dV=0.
\label{R-rho-integral}
\end{align}
Integrating \eqref{R-eqn} over $M$ by \eqref{f-laplacian-eqn} we get,
\begin{align}\label{R-rho-integral2}
\int_MR(R-\rho)\,dV=&-\frac{1}{2}\int_M<\nabla R,\nabla f>_g\,dV\notag\\
=&\frac{1}{2}\int_MR\Delta f\,dV\notag\\
=&\frac{n}{2}\int_MR(R-\rho)\,dV.
\end{align}
Since $n\ge 3$, by \eqref{R-rho-integral2},
\begin{equation}\label{R-rho-integral3}
\int_MR(R-\rho)\,dV=0.
\end{equation}
By \eqref{R-rho-integral} and \eqref{R-rho-integral3},
\begin{equation*}
\int_M(R-\rho)^2\,dV=0.
\end{equation*}
Hence $R\equiv\rho$ on $M$ and the theorem follows.
\end{proof}

\end{document}